\documentclass[12pt,reqno]{amsart}
\newcommand{\norm}[1]{\left \lVert #1 \right \rVert}

\usepackage{bm}
\usepackage{graphicx}
\usepackage{adjustbox}
\usepackage{amsthm,amssymb,amsmath}
\usepackage[T1]{fontenc}
\usepackage[utf8]{inputenc}
\usepackage[colorlinks=true, linkcolor=blue, urlcolor=blue, citecolor=blue, anchorcolor=blue]{hyperref}
\usepackage{xcolor}
\usepackage{mathrsfs}
\usepackage{enumitem}
\usepackage{setspace}
\usepackage{yfonts}
\usepackage{float}
\usepackage{mathtools}
\usepackage[all,cmtip]{xy}
\usepackage{todonotes}
\usepackage{tikz-cd}
\usepackage{relsize}

\newtheorem{manualtheoreminner}{Theorem}
\newenvironment{manualtheorem}[1]{%
\renewcommand\themanualtheoreminner{#1}%
  \manualtheoreminner
}{\endmanualtheoreminner}

\usepackage[backend=bibtex,style=alphabetic]{biblatex}

\tikzcdset{row sep/normal=50pt, column sep/normal=50pt}

\newtheorem{lemma}{Lemma}[section]

\newtheorem{corollary}[lemma]{Corollary}

\theoremstyle{definition}

\newtheorem{remark}[lemma]{Remark}

\numberwithin{equation}{section}

\begin{document}

\title[Symmetric power of higher dimensional varieties]{Symmetric power of higher dimensional varieties}

\author[A. Bansal]{Ashima Bansal}

\address{Indian Institute of Science Education and Research Tirupati, Srinivasapuram, Yerpedu Mandal, Tirupati Dist, Andhra Pradesh, India – 517619.}

\email{ashimabansal@students.iisertirupati.ac.in}

\author[S. Sarkar]{Supravat Sarkar}

\address{\scshape Fine Hall, Princeton, NJ 700108}

\email{ss6663@princeton.edu}

\author[S. Vats]{Shivam Vats}


\email{shivamvatsaaa@gmail.com}



\begin{abstract}
We study several properties of the symmetric power $S^mX$ of a smooth variety $X$. We describe the Picard and divisor class groups of $S^mX$ when $X$ is projective. We give a complete description of the stratification of $S^mX$ by iterated singular locus in terms of some combinatorial data regarding partitions of the integer $m.$ This gives a new viewpoint of a natural stratification of $S^mX$ by multiplicities.
 \end{abstract}
\maketitle

\section{Introduction}
\noindent Taking quotients under group actions is a ubiquitous topic in Mathematics. Various spaces parametrizing certain mathematical objects are constructed as group quotients. In algebraic geometry, there is the whole Geometric Invariant Theory, which defines and studies quotients of a variety or scheme under the action of an algebraic group. The most basic case of group quotient in algebraic geometry is the quotient of a quasi-projective variety under the action of a finite group. There are a lot of examples which come up in several areas. One of the most studied examples is that of weighted projective space, which is a quotient of a projective space under the action of a cyclic group. As weighted projective spaces are toric, they have been extensively studied in the literature, and much is known about them for a long time, see for example \cite{Do}.

However, there is another example of a group quotient which occurs very naturally. For a positive integer $m$, let $ S_m$ denote the permutation group of $m$ letters. Given a variety $X$, one defines the $m$'th symmetric power of $X$ by $S^mX\,=\,X^m/S_m,$ where $S_m$ acts on $X^m$ by permuting the factors. If $X$ is a projective variety over an algebraically closed field, the variety $S^mX$ parametrizes effective $0$-dimensional algebraic cycles on $X$ of degree $m$. Its elements can be written as $\sum_{i=1}^m x_i$ with $x_i\,\in\, X$. If $X$ is a smooth curve, then a remarkable property is that $S^mX$ is smooth for all $m$. For $X$ a smooth projective curve, the smooth projective variety $S^mX$ is very well studied in terms of the Abel-Jacobi map $S^mX\to$ Jac$^m(X)$, where Jac$^m(X)$ is the variety parametrizing degree $m$ line bundles on $X.$ This has connections to many other areas like topology and physics, as explained in \cite{BGZ}.

But if dim $X\,\geq\, 2$, $S^mX$ is singular. In this situation, the study of $S^mX$ is more delicate. It is an example of Chow variety and occurs naturally while studying the Hilbert scheme of points on a smooth projective variety via the Hilbert-Chow morphism, see \cite{fogarty1969truncated}, \cite{MFK}.

The goal of this article is to study several aspects of the symmetric power of smooth varieties. We have three results: Theorems 1 to 3. Among them, Theorems \ref{A} and \ref{C} are easy consequences of some previous works. The main new result in this paper is Theorem \ref{B}.

We work throughout over the field $\mathbb{C}$ of complex numbers. So, a variety can be thought of as a complex analytic space. All the varieties in this paper are assumed to be quasi-projective. Whenever we talk about the fundamental group, singular homology/cohomology of a variety, it is with respect to the analytic topology. Our first result describes the Picard group and class group of the symmetric power of smooth projective varieties. Throughout the paper, for a normal variety $Z$, Pic $Z$, Cl $Z$ stands for the Picard group and divisor class group of $Z$, respectively, and $K_Z\in$ Cl $Z$ is the canonical class.

\begin{manualtheorem}{1}\label{A}

Let $Y$ be a smooth projective variety of dimension $n\, \geq\, 2$. Let $m\,\geq\, 2$ be an integer and $ X\,=\, S^{m}Y$. Then we have the following:
\begin{enumerate}
\item[$(i)$] $\textnormal{Pic}\,X \,\cong\, \textnormal{Pic}\,Y \,\times\,$ \textnormal{NS(Alb}$(Y))$, where \textnormal{NS(Alb}$(Y))$ is the N{\'e}ron-Severi group of the Albanese variety of $Y$. 
\item[$(ii)$] The standard inclusion $ \textnormal{Pic}\, X\, \hookrightarrow\, \textnormal{Cl}\,X$ induces a direct sum $$\textnormal{Cl}\, X \,\cong\, \textnormal{Pic}\,X\, \oplus\, \mathbb{Z}/{2}\mathbb{Z}\, \cong\, \textnormal{Pic}\,Y\, \oplus\, \textnormal{NS(Alb($Y$))} \,\oplus\,\mathbb{Z}/{2}\mathbb{Z}.$$
\item[$(iii)$] Under the isomorphism in $(ii)$, $ K_{X}\, \in\, \textnormal{Cl}\,X$  corresponding to 
\begin{equation*}
 \begin{cases}
(K_{Y}, 0, \Bar{1}), & \quad \textnormal{if $n$ is odd},\\
(K_{Y}, 0, \Bar{0}), & \quad\textnormal{if $n$ is even}, 
\end{cases}
\end{equation*}
in $\textnormal{Pic}\,Y\, \oplus\, \textnormal{NS(Alb}(Y)) \,\oplus\, \mathbb{Z}/{2\mathbb{Z}}.$
         
\end{enumerate}
 \end{manualtheorem}

Part $(i)$ of the above theorem was already shown in \cite{Fo}. Our results are parts $(ii)$ and $(iii).$

 Next, we study the singular locus of a symmetric power more closely. For a reduced scheme $X$ of finite type over $\mathbb{C}$, let Sing $X$ be the singular locus of $X$ with reduced subscheme structure. For nonnegative integers $t$, define Sing$^t(X)$ inductively by Sing$^0(X)\,=\,X$, and Sing$^t(X)\, = \,$ Sing (Sing$^{t-1} X)$ for $t\,\geq\, 1$. So, $ \textnormal{Sing}^{t}(X)\setminus \textnormal{Sing}^{t+1}(X)$ is smooth for $t\,\geq\, 0$, and the irreducible/connected components of $ \textnormal{Sing}^{t}(X)\setminus \textnormal{Sing}^{t+1}(X)$ for varying $t$ gives a stratification of $X$ into locally closed smooth subvarieties. In this article, we completely describe this stratification for $X\,=\,S^mY$ for a smooth variety $Y$ in terms of some combinatorial data regarding partitions of the integer $m.$
 
To state our result, we need to introduce some notations. For a positive integer $m,$ let $[m]$ denote the set $\{1,2,\cdots , m\}$. A partition $\pi$ of $m$ is a tuple $\pi\,=\, ( a_{1},a_{2},\cdots, a_{k})$ of positive integers, with $a_{1}\, \geq\, a_{2}\, \geq\, \cdots  \geq\, a_{k}$ and $ \sum_{i=1}^{k}a_{i}\,=\,m$. We define $ |\pi|\,:=\, k.$ Given two partitions $ \pi^{\prime}\,=\, (a_{1},a_{2},\cdots, a_{k})$, $ \pi\,=\, (b_{1},b_{2},\cdots, b_{\ell})$ of $m$, a refinement of $\pi^{\prime}$ by $\pi$  is a surjective map $ [\ell] \,\xlongrightarrow{\,\,\,r\,\,\,}\, [k] $ such that for all $ i\, \in\, [k],$ $ \sum_{j \in r^{-1}(i)}b_{j}\,=\, a_{i} $. Two refinements $ r, r^{\prime}$ of $\pi^{\prime}$ by $ \pi$ are said to be equivalent if for all $ i\, \in\, [k] $, $ (b_{j})_{j\in r^{-1}(i)}, (b_{j})_{j\in (r^{\prime})^{-1}(i)} $ are same partitions  of $a_{i}.$ The symbol $\pi\xrightarrow{r}\pi'$ means $r$ is a refinement of $\pi'$ by $\pi$. We say $\pi\, \geq\, \pi^{\prime}$ if there is a refinement of $ \pi $ by $ \pi^{\prime}$, $\pi\, >\, \pi^{\prime} \hspace{2pt} \textnormal{if} \hspace{2pt} \pi\, \geq\, \pi^{\prime}\hspace{2pt} \textnormal{and}  \hspace{3pt} \pi^{\prime} \neq \pi.$

Now fix a positive integer $n\,\geq\, 2$. For a smooth variety $Y$ of dimension $n$, and $ \pi\,=\, (a_{1},\cdots,a_{k})$ a partition of a positive integer $m,$ let $ W_{\pi,m}(Y)\,=\, \{ \sum_{i=1}^{k}a_{i}y_{i}\, \in\, S^{m}Y \, \, | \, \, y_{1},y_{2},\cdots,y_{k} \in Y\}$, $ W_{\pi,m}^{\circ}(Y)\,=\, W_{\pi,m}(Y) \backslash \underset {{\pi^{\prime}\, >\, \pi}}{\bigcup} W_{\pi^{\prime},m}(Y)$. One easily sees that $W_{\pi, m}(Y)$ is a closed subvariety of $ S^{m}Y$. These subvarieties already appeared in {\cite[Section 2.1]{Sh}} and {\cite[Proposition 3]{Gir}}. Our result is the following.

\begin{manualtheorem}{2}\label{B}
Let $ m,n\, \geq\, 2$ be integers, $Y$ be a smooth variety of dimension $n$.
\begin{enumerate}
   
\item[$(i)$] For a partition $ \pi\,=\, (\underbrace{a_{1},\cdots,a_{1}}_{r_{1}\text{-times}}, \underbrace{a_{2},\cdots,a_{2}}_{r_{2}\text{-times}},\cdots, \underbrace{a_{\ell},\cdots,a_{\ell})}_{r_{\ell}\text{-times}}$ of $m$, where $ a_{1}\,>\,a_{2}\,>\, \cdots\,>\, a_{\ell}$, $ W_{\pi,m}(Y)$ is a closed subvariety of $ S^{m}Y$ of dimension $n|\pi|,$ and normalization $ \prod_{i=1}^\ell S^{r_{i}}Y$.
\vspace{6pt}
\item[$(ii)$] $ \textnormal{Sing}\, W_{\pi,m}(Y)\,=\, \underset{\pi^{\prime}\, >\,\pi} {\bigcup}W_{\pi^{\prime},m}(Y).$ In other words, $W_{\pi,m}^{\circ}(Y)$ is the smooth locus of $ W_{\pi,m}(Y).$

\vspace{6pt}
\item[$(iii)$] $ \textnormal{Sing}^{t}(S^{m}Y)\, =\, \underset{|\pi|\,=\, m\,-\,t}{\bigcup}W_{\pi,m}(Y)\, =\, \{ \sum_{i}y_{i} \in S^{m}Y \, | \, \textnormal{ there are  at\,most}\, m\,-\,t \hspace{2pt} \textnormal{ distinct}\, y_{i}\textrm{'s}\}$.

\vspace{10pt}
\item[$(iv)$] $ \textnormal{Sing}^{t}(S^{m}Y)\setminus\textnormal{Sing}^{t+1}(S^{m}Y)\,=\, \underset{|\pi|\,=\,m\,-\,t}{\bigsqcup}W^{\circ}_{\pi,m}(Y)$ \\

$ \hspace{135pt}=\,\{ \sum_{i}y_{i}\in S^{m}Y \medspace| \medspace \textnormal{ there}  \textnormal{ are exactly $m\,-\,t$ distinct $y_{i}$\textrm{'s}} \}$.
\end{enumerate}
\end{manualtheorem}

Next, we compute the discrepancy of the singularity of symmetric power, which is an invariant of a variety important in birational geometry (see \cite{KM} for the definition of discrepancy). 

\begin{manualtheorem}{3}\label{C}
Let $ m,n \,\geq\, 2$ be integers, $Y$ be a smooth variety of dimension $n$. Then the discrepancy of $S^{m}Y$ is given by:
$$ \textnormal{discrep}\medspace (S^{m}Y)\,=\,\begin{cases} 0, &\textnormal{if}\medspace\,n\,=\,2\\
\frac{1}{2}, & \textnormal{if}\medspace\, n\,=\,3\\
1, & \textnormal{if} \medspace\, n\,\geq\,4.
\end{cases}$$

So, $S^mY$ always has canonical singularities, and has terminal singularities if and only if $n\,\geq\, 3.$
\end{manualtheorem}

\section{Local study of symmetric power}
It will be convenient to use the notion of \text{germ} of a complex analytic set. See {\cite[Chapter 2]{Hu}} for the notion of germs. For a point $p$ on a complex analytic set $X$ we denote the germ of $X$ at $p$ by $\langle X, p\rangle$.

Fix a positive integer $n$. Let $m$ be a positive integer and $Y$ a smooth variety of dimension $n$. Recall that for a partition $ \pi\,=\, (b_{1},\cdots,b_{k})$ of $m,$  $ W_{\pi,m}(Y)\,=\, \{ \sum_{i=1}^{k}b_{i}y_{i} \in S^{m}Y\,{|}\, y_{1},\cdots, y_{k}\in Y\}$, $ W_{\pi,m}^{\circ}(Y)\,=\, W_{\pi,m}(Y) \backslash \underset {{\pi^{\prime} \,>\, \pi}}{\bigcup} W_{\pi^{\prime},m}(Y)    $. Let  $W_{\pi,m}^{g}\,=\, \langle W_{\pi,m}(\mathbb{A}^{n}),0 \rangle$, $ X_{m}^{g}\,=\, \langle S^{m}\mathbb{A}^n, 0 \rangle, $  and $W_{\pi,m}^{g}\,=\, \langle W_{\pi,m}(\mathbb{A}^n),0 \rangle $, a subgerm of $ X_{m}^{g}.$ Here by a subgerm of a complex analytic set $X$ at a point $p\in X$, we mean germ of a closed complex analytic subset of $X$ passing through $p,$ and $0\in S^m \mathbb{A}^n$ corresponds to the effective $0$-cycle of degree $m$ supported on $\{0\}$. Let $(\underline{1})$
be the partition $ (\underline{1})\,=\, \underbrace{(1,1,\cdots,1)}_{m\text{-times}}$ and $(m)$ be the partition by a single part $m$ of $m$.
\vspace{7pt}

\begin{lemma}\label{prelim} In the previous notations we have the following:
\begin{enumerate}
\item[$(i)$] $W_{\pi, m}(Y)$ is a closed subvariety of $ S^{m}Y$ of dimension $n \hspace{1pt}|\pi|.$

\item[$(ii)$] $ W_{(\underline{1}),m}(Y)\,=\, S^{m}Y$, $ W_{(m),m}(Y)\, \cong\, Y.$ 
    
\item[$(iii)$] $ W_{\pi_{{1}},m}(Y)\, \subseteq\, W_{\pi_{2},m}(Y)$ if and only if $\pi_{2}\, \leq\, \pi_{1}$.\
    
\item[$(iv)$] $W_{\pi,m}(Y)\,=\,\underset{\pi^{\prime}\, \geq\, \pi}{\bigsqcup}\, W^{\circ}_{\pi^{\prime},m}(Y)$ as sets. 
\item[$(v)$] $W_{\pi_{1},m}(Y)\, \cap\, W_{\pi_{2},m}(Y)\,=\,\underset{\pi\, \geq\, \pi_{2}}{\underset{\pi\,\geq\,\pi_{1}}{\bigcup}} W_{\pi,m}(Y)$. 

\item[$(vi)$] If $ \pi\,=\, \underbrace{(a_{1},\cdots,a_{1}}_{r_{1}\text{-times}}, \underbrace{a_{2},\cdots, a_{2}}_{r_{2}\text{times}}, \cdots, \underbrace{a_{\ell},\cdots, a_{\ell}}_{r_{\ell}\text{-times}})$, where $ a_{1}\,>\, a_{2}\,>\, \cdots\,>\, a_{\ell}$, then the normalization of $W_{\pi,m}$ is $ \prod_{i=1}^{l}S^{r_{i}}Y.$
\end{enumerate} 
\end{lemma}

\begin{proof}
The proof of $(ii)-(v)$ is straightforward. For $(i)$, note that the map $$Y^{k}\,\longmapsto\, S^{m}Y $$ 
$$ (y_{1}, y_{2}, \cdots,y_{k})\, \longmapsto\, \sum_{i}a_{i}y_{i} $$ is finite and
has image $ W_{\pi,m}(Y)$. For $(vi)$, consider the map 
    $$ \prod_{i=1}^{\ell} S^{r_{i}}Y\, \xlongrightarrow{\,\,\,\alpha\,\,\,}\, W_{\pi,m}Y $$ 
    $$ ( \sum_{j=1}^{r_{i}}y_{ij})_{i} \,\longmapsto\, \sum_{i=1}^{l}\sum_{j=1}^{r_{i}}a_{i}y_{ij}.$$
    
    Note that $\alpha$ is surjective, finite, and $| \alpha^{-1}(p)|\, =\, 1$, for a general point $ p \,\in\, W_{\pi,m}(Y).$ So, $ \alpha $ is finite birational. Since $\prod_{i=1}^{\ell}S^{r_{i}}Y$ is normal, it follows that $ \alpha$ is a normalization map.

\end{proof}
\begin{lemma}\label{local structure}
    Let $ p\,\in\, S^{m}(Y)$ be given by $ p\,=\, \sum_{i=1}^{k}a_{i}y_{i},$ where $y_{i} $' s are distinct, and $ a_{1}\,\geq\, a_{2}\,\geq\, \cdots\, \geq\, a_{k}.$  Let $ \pi\,=\,(a_{1},a_{2},\cdots,a_{k})$. Then there is an isomorphism of germs 
    $\prod_{i=1}^{k} X_{a_{i}}^g\,\xrightarrow{\,\,\,\phi\,\,\,}\, \langle S^{m}Y, p \rangle$ such that the following holds:
    
    For any $\pi^{\prime}\,=\, (b_{1}, b_{2},\cdots,b_{\ell})$ such that $ \pi\, \geq\, \pi^{\prime}.$ Then we have $$ \phi^{-1}\langle W_{\pi^{\prime},m}(Y), p\rangle\, =\, \bigcup_{\pi^{\prime}\xrightarrow{r} \pi} \prod_{i=1}^{k} W_{r_{(i)},a_{i}}^{g}.$$ Here, $r$ runs over all refinements of $ \pi^{\prime}$ by $\pi,$ up to equivalence, and $r_{(i)}$ denotes the partition $ (b_{j})_{j\in r^{-1}(i)}$ of $a_i.$

\end{lemma}
\begin{proof}
First, we define the isomorphism $\phi$.  Since $Y$ is smooth, for each $y\in Y$ we have an isomorphism of germs: 
$$ \langle T_{y}Y,0 \rangle\, \xlongrightarrow{\,\cong\,}\,\langle Y, y \rangle .$$ For each $ y\,\in\, Y$, denote this isomorphism by $$ v\, \longmapsto\, y\,\oplus\, v.$$

Let $ \underline{q}\,=\, (\underbrace{y_{1},\cdots,y_{1}}_{a_{1}\text{-times}}, \underbrace{y_{2},\cdots,y_{2}}_{a_{2}\text{-times}},\cdots, \underbrace{y_{k},\cdots,y_{k})}_{a_{k}\text{-times}}\,\in\, Y^{m}.$ Under the action of $S_{m}$ on $Y^{m}$, $ H\,:=\,\textnormal{stab}(\underline{q})\,=\,S_{a_{1}} \,\times\, S_{a_{2}}\,\times\, \cdots\, \times\, S_{a_{k}},$ as $y_{i} \hspace{1pt}$'s are distinct. 
We have $ \langle Y^{m}, \underline{q} \rangle \,\cong\,\underset{i} {\prod} \langle Y^{a_{i}}, (y_{i},y_{i}, \cdots, y_{i})\rangle \,\cong\, \underset{i} {\prod}\langle(T_{y_{i}}Y)^{a_{i}}, \underline{0}\rangle $, and the action of $ H\,=\, \underset{i} {\prod}S_{a_{i}}$ on $ \underset{i} {\prod} \langle (T_{y_{i}}Y)^{a_{i}}, \underline{0}\rangle$ is just the product action. So, we have isomorphisms

$$\hspace{-1.3cm} \langle S^{m}Y,p \rangle\, \cong\, \langle Y^{m}, \underline{q} \rangle \big{/}H$$

    \hspace{185pt}$ \cong\, \underset{i} {\prod} \langle (T_{y_{i}}Y)^{a_{i}}, \underline{0}\rangle\big{/}S_{a_{i}}$

    $ \hspace{185pt}\cong\, \underset{i} {\prod}\langle S^{a_{i}}\mathbb{A}^n,0 \rangle$

    $ \hspace{185pt}\cong\, \underset{i} {\prod}X_{a_{i}}^{g}$.
    
Call this composition of isomorphisms $\phi\,:\,\prod_{i=1}^{k} X_{a_{i}}^g\longrightarrow \langle S^{m}Y, p \rangle$.
So, $ \phi$ is given by 
\begin{equation}\label{phi}
    \phi ((v_{i1}+ \cdots + v_{ia_{i}})_{i})\,=\, \sum_{i}(y_{i}\,\oplus\, v_{i1}+ \cdots + y_{i}\,\oplus\, v_{ia_{i}}).
\end{equation} 
Here, $v_{ij}\in T_{y_{i}}Y,$ and $ v_{i1}+\cdots +v_{ia_{i}}$ is considered as a zero cycle in $ T_{y_{i}}Y$ (it is not the vector addition).

Now we show that $\phi$ has the desired property.
Choose metrics on each $T_{y_{i}}Y$. Let $\epsilon \,>\,0$ be such that $ y_{i}\,\oplus\, v$ is defined for all $i$ and $\norm{v}\, <\, \epsilon, $ and $ y_{i}\,\oplus\, B_{\epsilon}(0)$' s are disjoint. Let $$\widehat{X_{a_{i}}^{g}}\,=\, \Big \{ \sum_{j=1}^{a_{i}}v_{ij} \in S^{a_{i}}(T_{y_{i}}Y) \hspace{1pt}\,\big{|}\, \hspace{1pt}  \norm{v_{ij}}\,<\, \epsilon  \hspace{0.3cm} \textnormal{for all}\,\, j \Big \}.$$ 
So, $\phi$ is represented by a holomorphic open map 
    $\underset{i} {\prod}\widehat{X_{a_{i}}^{g}}\, \xlongrightarrow{\,\,\,\widehat{\phi}\,\,\,}\, S^{m}Y$, which is given by the same formula as in  \eqref{phi}, and is an isomorphism onto its image, which we can call $ \widehat{S^{m}Y}.$ Let $ \widehat{W}_{r_{(i)},a_{i}}^{g}\,=\, \widehat{X_{a_{i}}^{g}} \,\cap\, W_{r_{(i)},a_{i}}^{g} $. It suffices to show $$ {\widehat{\phi}}^{-1}(W_{\pi^{\prime},m}(Y))\,=\, \bigcup_{\pi^{\prime}\xrightarrow{r} \pi} \prod_{i=1}^{k} \widehat{W}_{r_{(i)},a_{i}}^{g} .$$
     \underline{$ {\widehat{\phi}}^{-1}( W_{\pi^{\prime},m}(Y)) \supseteq\underset{\pi^{\prime}\xrightarrow{r} \pi}{\bigcup} \prod_{i=1}^{k} \widehat{W}_{r_{(i)},a_{i}}^{g}  $:} 
    
    Let  $ \pi^{\prime}\,=\, (b_{s})_{s},$ and $ \pi^{\prime}\xrightarrow{r} \pi$ be a refinement. Let $ \underline{v}\,=\,(\underset{s\in r^{-1}(i)}{\sum}b_{s}v_{is})_{i}$ be an element of $ \prod _{i=1}^k \widehat{W}_{r_{(i)},a_{i}}^{g}$, where $ v_{is}\in T_{y_{i}}Y.$ We have 
    $$ \widehat{\phi}(\underline{v})\,=\, \sum_{i} \sum_{s\in r^{-1}(i)}b_{s}y_{i}\, \oplus\, v_{is}\,=\, \sum_{s}b_{s}y_{r(s)}\,\oplus\, v_{r(s),s} \in W_{\pi^{\prime},m}(Y).$$
\underline{$ \widehat{\phi}^{-1}( W_{\pi^{\prime},m}(Y)) \,\subseteq\, \underset{\pi^{\prime}\xrightarrow{r} \pi}{\bigcup} \prod_{i=1}^{k} \widehat{W}_{r_{(i)},a_{i}}^{g}  $:}

    Let $ v_{ij} \in T_{y_{i}}Y$, for $1\,\leq\, i\,\leq\, k,\quad 1\,\leq\, j\, \leq\, a_i,$ be such that $ \underline{v}\,=\, (\sum_{j}v_{ij})_{i} \in \widehat{\phi}^{-1}(W_{\pi^{\prime},m}(Y)),$ i.e., $$ \sum_{i}\sum_{j}y_{i}\,\oplus\, v_{ij} \in W_{\pi^{\prime},m}(Y).$$ So, there are $z_{s}$'s in $Y$ such that 
    $$ \sum_{i=1}^{k} \sum_{j=1}^{a_{i}} y_{i}\,\oplus\, v_{ij}\,=\,\sum_{s=1}^l b_{s}z_{s}.$$
    Since  $ y_{i}\,\oplus\, B_{\epsilon}(0)$'s are disjoint, for each $s$ we have unique $ r(s)\in [k]$ such that $ z_{s} \in y_{r({s})}\, \oplus\, B_{\epsilon}(0).$ This gives a map $ r : [\ell]\longrightarrow [k],$ we have 
    
    \begin{equation}\label{sum}
        \sum_{i=1}^{k}\sum_{j=1}^{a_{i}} y_{i}\,\oplus\, v_{ij}\,=\, \sum_{i=1}^{k}\sum_{s\in r^{-1}(i)}b_{s}z_{s}.
    \end{equation}
    So, $$ \sum_{i=1}^k(\sum_{j=1}^{a_{i}} y_{i}\,\oplus\, v_{ij}\,-\, \sum_{s\in r^{-1}(i)}b_{s}z_{s})\,=\,0.$$
    The $0$-cycle $\sum_{j=1}^{a_{i}} y_{i}\,\oplus\, v_{ij}\,-\, \sum_{s\in r^{-1}(i)}b_{s}z_{s}$ is supported in $y_{i}\,\oplus\, B_{\epsilon}(0).$
    Since $ y_{i}\,\oplus\, B_{\epsilon}(0)$'s are disjoint, we must have 
    \begin{equation}\label{exponential}
\sum_{j=1}^{a_{i}}y_{i}\,\oplus\, v_{ij}\,=\, \sum_{s\in r^{-1}(i)}b_{s}z_{s}\quad \textnormal{for all} \hspace{3pt}i. 
\end{equation}

Taking degrees, we get $ a_{i}\,=\,\underset{s\in r^{-1}(i)}{\sum}b_{s}.$ Thus, $ r$ is a refinement. Since $$ B_{\epsilon}(0)\, \longrightarrow\, y_{i}\,\oplus\, B_{\epsilon}(0)$$
$$ v \longmapsto y_{i}\,\oplus\, v$$
is an isomorphism,  there are unique $ w_{s}\,\in\, B_{\epsilon}(0)$ such that $ z_{s}\,=\, y_{r(s)} \,\oplus\, w_{s}.$ Now, \eqref{exponential} gives $ \sum_{j=1}^{a_{i}} v_{ij}\,=\, \sum _{s\in r^{-1}(i)} b_{s}w_{s}.$ So, $ \sum_{j}v_{ij}\, \in\, \widehat{W}_{r_{(i)},a_{i}}^{g}$. Hence $ \underline{v}\, \in\, \prod_{i=1}^{k}\widehat{W}_{r_{(i)},a_{i}}^{g}$.
\end{proof}

As an immediate corollary we obtain
\begin{corollary}\label{deepest}
If $ p\,\in\, W_{(m),m}(Y),$ then $ \langle W_{\pi,m}(Y),p \rangle\, \cong\, W_{\pi,m}^{g}$.
\end{corollary}
We also have the following: 
\begin{corollary}\label{gorenstein}
Let $ m\,,\,n\, \geq\, 2$, and $Y$ be a smooth variety of dimension $n$. Then,  $K_{S^{m}Y}$ is Cartier if and only if $n$ is even. 
\end{corollary}
\begin{proof}
$K_{S^{m}Y}$ is Cartier $$\Longleftrightarrow  \langle S^{m}Y, p\rangle\, \text{is Gorenstein for all }p\, \in\, S^{m}Y$$ $$\hspace{2cm} \hspace{0.7cm}\Longleftrightarrow X_{a_{i}}^{g} \text{ is Gorenstein for all } 2\,\leq\, a_{i}\,\leq\, m \text{ (by Lemma \ref{local structure}})$$ $$\hspace{4cm}\Longleftrightarrow \text{ for all }2\leq a_{i}\,\leq\, m,\text{ the action of } S_{a_{i}} \text{ on } (k^{n})^{a_{i}}\text{ by permutation}$$ $$\text{\hspace{0.9cm}is through a map } S_{a_{i}} \,\longrightarrow\, \textrm{SL}(na_{i},k)$$ $$\hspace{-5.2cm}\Longleftrightarrow\, n \text{ is even}.$$

The third equivalence follows from \cite{watanabe1974certain} and \cite{watanabe1974certain2}, noting that if $a_i\,\geq\, 2$, $S_{a_i}$ acts on $(k^{n})^{a_{i}}$ without pseudo-reflection.
\end{proof} 
\begin{remark}
    In {\cite[Lemma 7.1.7]{BK}}, it is proved that $S^m Y$ is always Gorenstein. But this statement is incorrect, and the proof of {\cite[Lemma 7.1.7 (ii)]{BK}} has a mistake: in they assume that the isomorphism $K_Y^{\boxtimes m}\,\cong\, K_{Y^m}$ is $S_m$-equivariant, but actually it is not when $n$ is odd.
\end{remark}

\begin{lemma}\label{Singular}
$W_{\pi,m}^{g}$ is singular if $ |\pi|\, >\,1$.
\end{lemma}

\begin{proof}
By Corollary \ref{deepest}, it suffices to show that for $ |\pi|\,>\,1,$ $ W_{\pi,m}(Y)$ is singular along a point $ p\,\in\, W_{(m),m}(Y)$ for a smooth variety $Y.$ Suppose $ W_{\pi,m}(Y)$ is smooth at $p$. Let $  \pi\,=\, (\underbrace{a_{1},\cdots,a_{1}}_{r_{1}\text{-times}}, \underbrace{a_{2},\cdots,a_{2}}_{r_{2}\text{-times}},\cdots, \underbrace{a_{k},\cdots,a_{k})}_{r_{k}\text{-times}}$, where $a_1\,>\,a_2\,>\,\ldots\,>\,a_k.$ By Lemma \ref{prelim}$(vi)$, we have the normalization map $$ \prod_{i}S^{r_{i}}Y\, \xlongrightarrow{\,\,\,\alpha\,\,\,}\, W_{\pi,m}(Y).$$ Let $ y\,\in\, Y$ be such that $ p\,=\,my.$ We have $ \alpha^{-1}(p)\, =\, (r_{1}y, r_{2}y,\cdots,r_{k}y).$ Since $ W_{\pi,m}(Y) $ is smooth at $p$, the normalization $ \prod_{i}S^{r_{i}}Y$ must be smooth at $ \alpha^{-1}(p).$ Hence, each $S^{r_{i}}Y$ must be smooth at $ r_{i}y.$ So, $ r_{i}\, = \,1$ for all $\hspace{2pt}i$.

Also, as $ W_{\pi,m}$ is smooth at $p,$ there is a smooth neighborhood $V$ of $p$ in $ W_{\pi,m}(Y),$ hence $$ \alpha^{-1}(V)\, \xlongrightarrow{\,\,\,\alpha\,\,\,}\, V$$  is an isomorphism. Let $f\,:\, Y^m\,\to\, S^mY$ be the projection. So, we have a commutative diagram 
\begin{center}
\begin{tikzcd}

f^{-1}V\arrow[rd, "g"] \arrow[d,"f"]
 \\
V \arrow[r, "\alpha^{-1}" ]
& |[, rotate=0]|  Y^{k}
\end{tikzcd},
\end{center}
 where $g$ is $S_{m}$-invariant. Let $ \underline{z}\,=\, (\underbrace{y,\cdots,y}_{m\text{-times}})\,\in\, f^{-1}(V)$ and $\underline{w}\,=\, (\underbrace{y,\cdots,y}_{k \text{-times}})\,\in\, Y^{k}.$ So, $f(\underline{z})\,=\,p$, $g(\underline{z})\,=\, \underline{w}$. 
 
 For a closed point $z$ on a variety $Z$, denote the Zariski tangent space of $Z$ at $z$ by $T_zZ$. Let $ U\,=\, T_{y}Y.$ 
Denote an element of $Y^{m}$ by $(y_{ij})_ {\underset{1\,\leq\, j\, \leq\, a_{i}}{1\,\leq\, i \,\leq\, k}}$
and an element of $ U^{m}$ by 
$ (v_{ij})_{\underset{1\,\leq\, j\,\leq\, a_{i}}{ 1\,\leq\, i\, \leq\, k}},$ where $y_{ij}\,\in\, Y, v_{ij}\,\in\, U.$ So, $S_m$ is regarded as the group of permutations of the index set $\{(i,j)\,\in\, \mathbb{Z}^2 \,|\, 1\,\leq\, i\,\leq\, k,\quad 1\,\leq\, j \,\leq\, a_i\}$ of $m$ elements.
Let $$ D\,=\, \{ (y_{ij})\,\in\, Y^{m} \hspace{4pt}| \hspace{4pt}y_{i1}\,=\,\cdots \,=\, y_{ia_{i}} \hspace{3pt} \textnormal{for all} \hspace{3pt} i\} \,\subseteq\, f^{-1}(W_{\pi,m}(Y)).$$

We have: $ T_{\underline{z}} f^{-1}(V)$ is an $S_{m}$-invariant subspace of $ T_{\underline{z}}Y^{m}\,=\, U^{m},$ containing 

$$ T_{\underline{z}}D\,=\, \{ (v_{ij})\,\in\, U^{m} \hspace{3pt} | \hspace{3pt} v_{i1}\,=\,\cdots \,=\,v_{ia_{i}} \hspace{3pt} \textnormal{for all}  \hspace{3pt} i \},$$ and hence containing $ \sum_{\sigma\,\in\, S_{m}}\sigma(T_{\underline{z}}D).$

$\underline{\textnormal{Claim}}:$ $ T_{\underline{z}}f^{-1}(V)\,=\, U^{m}.$
\begin{proof}
It suffices to show that $\sum_{\sigma\, \in\, S_{m}} \sigma(T_{z}D)\, =\, U^{m}.$ Suppose not. Then there is $ 0 \neq \lambda \,=\, (\lambda_{ij})_{i,j}\, \in\, (U^{m})^{\ast}\,=\, (U^{\ast})^{m}$ such that $ \lambda (\sigma(T_{\underline{z}}D))\, =\,0$ for all $\sigma\, \in\, S_{m}.$ So for all $\sigma\, \in\, S_{m}$, 
     $$ \lambda \circ \sigma \,\in\, U_{2}\,:=\, \{ \beta\, \in\, (U^{m})^{\ast} \hspace{3pt} | \hspace{3pt} \beta (T_{\underline{z}}D)\,=\,0 \}.$$
Clearly,
     $$ U_{2}\,=\, \{ (\beta_{ij})_{i,j}\,\in\, (U^{\ast})^{m} \hspace{3pt} | \hspace{3pt} \sum_{j}\beta_{ij}\,=\,0  \hspace{3pt} \textnormal{for all} \hspace{3pt} i \}.$$
     Note that, $ \lambda \circ \sigma\, =\, (\lambda _{\sigma(i,j)})_{i,j}.$ Let $ 1 \,\leq\, i,i^{\prime}\, \leq\, k$ be such that $ i \neq i^{\prime},$ and let $ 1 \,\leq\, j \,\leq\, a_{i}$ and $1\,\leq\, j^{\prime}\, \leq\, a_{i^{\prime}}$ be integers . Let $ \sigma\, \in\, S_{m}$ be the transposition of $ (i,j)$ and $(i^{\prime},j^{\prime}) .$ Since $ \lambda, \lambda \circ \sigma \, \in\, U_{2}$, we have $$ \sum _{1\, \leq\, s\, \leq\, a_{i}} \lambda_{i,s}\,=\,0 \,=\, \sum_{1\,\leq\, s \,\leq\, a_{i}, s\neq j} \lambda_{i,s} +\lambda_{i^{\prime},j^{\prime}}.$$ Thus, $ \lambda_{i,j}\, =\, \lambda_{i^{\prime},j^{\prime}}.$ Since, this holds for all $(i,j),(i^{\prime},j^{\prime})$ with $ i \neq i^{\prime}$, and we have $ k\,>\,1$, it follows that  all $ \lambda_{i,j}$' s are equal. However, since $ \sum_{j}\lambda_{i,j}\,=\,0 $ for all $i$, we must have $ \lambda_{i,j}\,=\,0 $ for all $ i,j$, a contradiction to $ \lambda \,\neq\, 0 $.   
     This completes the proof of the claim.
\end{proof}     

For $ \underline{y}\,=\, (y_{ij})\, \in\, D\, \cap\, f^{-1}V,$ we have $ \alpha g (\underline{y})\,=\, f(\underline{y})\,=\, \sum_{i}a_{i}y_{i1}\,=\, \alpha\big( (y_{i1})_{i}\big).$ So we have  $g(\underline{y})\,=\, (y_{i1})_{i}$ for all $\underline{y}\,\in\, D\, \cap\, f^{-1}(V).$ Hence, $ dg_{\underline{z}}((v_{ij}))\,=\, ((v_{i1}))_{i}$ for all $(v_{ij})\, \in\, T_{\underline{z}}D.$  If $ U^{m} \xlongrightarrow{p} U$ is the projection onto first factor, and  $ \lambda\, =\, p\circ dg_{\underline{z}},$ then $ \lambda({\underline{v}})\,=\, v_{1},$ where $$ \underline{v}\,=\,(\underbrace{v_{1},\cdots,v_{1}}_{a_{1}\text{-times}}, \underbrace{v_{2},\cdots,v_{2}}_{a_{2}\text{-times}},\cdots, \underbrace{v_{k},\cdots,v_{k})}_{a_{k}\text{-times}}\, \in\, U^{m}. $$ Also, since $g$ is $S_{m}$-invariant, so is $ dg_{\underline{z}},$ and so is $\lambda.$  

Let $ \lambda_{i,j} \,\in\, \textnormal{Hom}(U,U)$ be such that $ \lambda\, =\,( \lambda_{i,j})\,\in\, \textnormal{Hom}(U^{m},U) \,=\, \textnormal{Hom}(U,U)^m.$ Since $ \lambda$ is $S_{m}$-invariant, we have $ \lambda_{\sigma(i,j)}\,=\, \lambda_{i,j}$ for all $i,j$, and for all $ \sigma\, \in\, S_{m}.$ Thus, all $ \lambda_{i,j}$' s are same, say $ \mu\,=\, \lambda_{i,j}\hspace{2pt}\textnormal{for all} \hspace{2pt} i,j$.  So, $ \lambda\big( ( u_{ij}) \big)\,=\, \mu \big( \sum_{i,j}u_{ij} \big)$ for all $ (u_{ij})\, \in\, U^{m}.$ 

Now choose $v_{1} \neq 0$ in $U$. Since $ k\,>\,1$, we have 

$$ v_{1}\,=\, \lambda (\underbrace{v_{1},\cdots, v_{1}}_{a_{1} \text{-times}},0,\cdots,0)=\mu(a_{1}v_{1})\,= \,\lambda(\underbrace{0,\cdots, 0}_{a_{1} \text{-times}} ,\underbrace{\frac{a_{1}}{a_{2}}v_{1},\cdots, \frac{a_{1}}{a_{2}}v_{1}}_{a_{2} \text{-times}},0,\cdots,0)\,=\,0,$$ which is a contradiction. This proves the lemma.  
\end{proof}
\section{Picard and divisor class groups of symmetric power}

In this section, we prove Theorem \ref{A}.

\textit{Proof of Theorem $\ref{A}$}: Let $ Y^{m} \xlongrightarrow{f} X $  be the quotient map.
\begin{enumerate}
\item[$(i)$] This is already shown in \cite{Fo}.
\item[$(ii)$] Let 
\begin{equation*}
 U\,=\, \{ \sum_{i}y_{i}\,\in\, X \medspace | \medspace y_{i} \textnormal{' s are distinct} \},
\end{equation*}
  $\widetilde{U}\,=\, f^{-1}(U)\,=\, \{(y_{1},\cdots, y_{m})\,\in\, Y^m \medspace | \medspace y_{i}\textnormal{' s are distinct} \} .$ Since $S_m$ acts freely on $\widetilde{U}$, $f$ is finite {\'e}tale over $U.$ Let  $\textnormal{Pic}\,(\widetilde{U},S_{m})$, $(\textnormal{Pic}\, \widetilde{U})^{S_{m}}$ be the groups of $S_m$-linearized line bundles and $S_m$-invariant line bundles on $\widetilde{U}$, respectively. Since codimension of $ X\setminus U$ in $ X$ is $n\, \geq\, 2,$ we have Cl$\,X\,$\,=\,Pic $U$\,=\,$\textnormal{Pic}\,(\widetilde{U},S_{m})$ by {\cite[Chapter 2, Section 7, Proposition 2]{mumford1974abelian}}. Let $\textnormal{Pic}\,(\widetilde{U},S_{m})\xlongrightarrow{r} (\textnormal{Pic}\, \widetilde{U})^{S_{m}}$ b the map which forgets the $S_m$-linearization. We have an exact sequence,

\begin{equation}\label{git}
    1 \longrightarrow \textnormal{Hom}(S_{m},H^{0}(\widetilde{U}, \mathcal{O}_{\widetilde{U}}^{\ast})) \longrightarrow \textnormal{Pic}\,(\widetilde{U},S_{m})\xlongrightarrow{r} (\textnormal{Pic}\, \widetilde{U})^{S_{m}}.
 \end{equation}

Note that the codimension of $ Y^{m}\setminus \widetilde{U}$ in $ \widetilde{Y}$ is $n \,\geq\, 2.$
So, $ \textnormal{Pic}\,\widetilde{U}\,=\, \textnormal{Pic}\,(Y^{m}).$ So, we can identify $ \big(\textnormal{Pic}\, \widetilde{U})^{S_{m}} \,\cong\, \textnormal{Pic}\, X$ by {\cite[Lemma 3.5 and Proposition 3.6]{Fo}}. Let $ j\,:\, U\hookrightarrow X$ be the inclusion. So, $ \textnormal{Pic}\,X \xhookrightarrow{j^*} \textnormal{Pic}\, U $ is the standard map $ \textnormal{Pic}\, X \hookrightarrow  \textnormal{Cl}\, X.$ Also note that $r\circ j^*$ is the identity map on Pic $X$. So, \eqref{git} is exact on the right and $j^* $ splits the sequence. Finally, note that since the abelianization of $S_m$ is $\mathbb{Z}/2 \mathbb{Z}$, we have $\textnormal{Hom}(S_{m},H^{0}(\widetilde{U}, \mathcal{O}_{\widetilde{U}}^{\ast}))\,\cong\, \textnormal{Hom}(\mathbb{Z}/2 \mathbb{Z},H^{0}(\widetilde{U}, \mathcal{O}_{\widetilde{U}}^{\ast}))\,\cong\, \mathbb{Z}/2 \mathbb{Z}$.
\end{enumerate}
\begin{enumerate}
\item[$(iii)$] Since  $ \widetilde{U}\xlongrightarrow{f} U$ is finite {\'e}tale, we have $ f^{\ast}K_{U}\,=\, K_{\widetilde{U}}\,=\, {K_{Y^{m}}}\mid_{\widetilde{U}} \,=\,K_{Y}^{\boxtimes m} \mid_{\widetilde{U}}.$ Let Pic$^0Y$ denote the Picard variety of $Y$. Under the isomorphism $$ \text{Pic}(Y^m)\,\cong\, \text{Pic}(Y)^m\,\times\, \text{Hom(Alb}\, Y,  \text{Pic}^0Y)^{\binom{m}{2}}$$ as in \cite{Fo}, $K_{Y}^{\boxtimes m}\,\in\, \text{Pic}(Y^m)$ corresponds to $(K_Y, K_Y,...,K_Y)\,\times\, \underline{0}$. So, identifying $(\textnormal{Pic}\, \widetilde{U})^{S_{m}}\,\cong\, \text{Pic}(Y^m)^{S_{m}}$, and Cl $X\,=\,\textnormal{Pic}\,(\widetilde{U},S_{m})$ as in the proof of $(ii)$, we get $r(K_X)\,=\,(K_Y, 0)$ in \eqref{git}. So, there exists $\epsilon\, \in\, {\mathbb{Z}/2 \mathbb{Z}}$ such that under the isomorphism of $(ii)$, $K_{X}$ corresponding to $(K_{Y}, 0, \epsilon).$ Finally, note that $ \epsilon\,=\,0\text{ iff } K_{X}$ is Cartier iff $n$ is even (by Corollary \ref{gorenstein}). 
\end{enumerate}
\begin{remark}
    In this remark we give a large class of singular rational varieties with a polarised endomorphism which are not toric. This shows that the answer to {\cite[Question 4.4]{Fa}} is negative if one drops the assumption of smoothness of $X$, even if one assumes $X$ is rational. Let $Y$ be a smooth projective toric variety, $m, k\geq 2$ integers, $Y\hookrightarrow S^m Y$ the diagonal embedding. Scaling of the fan of $Y$ gives an endomorphism $f$ of $Y$ such that $f^* L=L^k$ for all line bundles $L$ on $Y$. $f$ induces an endomorphism $f_m$ of $S^mY$, whose restriction to $Y$ is $f$. As $\text{Pic}(Y)$ is free abelian, by Theorem \ref{A} the restriction gives an injection $\text{Pic}(S^mY)\to\text{Pic}(Y)$. So, $f_m^* M=M^k$ for all line bundles $M$ on $Y$. So, $f_m$ is a polarised endomorphism of $S^mY.$ Since $Y$ is rational, by \cite{Ma}, $S^m Y$ is rational.

    On the other hand $S^m Y$ is never toric. In fact, we show that dim Aut$^0(S^m Y)=$dim  Aut$^0(Y)$. Let $G$ be the closed subgroup of Aut $Y^m$ consisting of all automorphisms which descend to $S^m Y$. By {\cite[Proposition 9 and 12]{belmans2020automorphisms}}, the natural map $G\to $Aut $S^mY$ is surjective (The proof of {\cite[Proposition 9 and 12]{belmans2020automorphisms}} is valid for every smooth projective variety $X$, not only surfaces). So, dim Aut$^0(S^m Y)\leq$dim $G^0$. On the other hand, $G^0\subset G\cap Aut^0(Y^m)$. Since any automorphism in Aut$^0(Y^m)$ preserves the factors, any automorphism in $G\cap \text{Aut}^0(Y^m)$ must be induced by a single automorphism of $Y$. This shows dim $G^0\leq$ dim Aut$^0(Y)$. Clearly an automorphism of $Y$ induces an automorphism of $S^m Y$, so we have dim Aut$^0(Y)\leq $dim Aut$^0(S^m Y)$. Hence, we get the desired equality of dimensions.
\end{remark}
\section{Stratification of symmetric power}
In this section, we prove Theorem \ref{B}.

\textit{Proof of Theorem $\ref{B}$}:
   Part $(i)$ is already proven in Lemma \ref{prelim}.
    
$(ii)$  If $ p\,\in\, W_{\pi,m}^{\circ}(Y),$ then by Lemma \ref{local structure} we have, $$ \langle W_{\pi,m}(Y),p \rangle \,\cong\, \prod_{i=1}^{l} W_{(a_{i}),a_{i}}^{g} \,\cong\, \prod_{i=1}^{l} \langle \mathbb{A}^n, 0 \rangle $$  
     is smooth. If $ p \,\in\, W_{\pi^{\prime},m}(Y)$ for some $ \pi^{\prime}\,>\,\pi $, let $ \pi \xlongrightarrow{r}\pi^{\prime}$  be a refinement. We have $ \textnormal{dim}\medspace \prod_{i=1}^{l} W_{r_{(i)},a_{i}}^{g}\,=\,n |\pi|\,=\, \textnormal{dim} \medspace \langle W_{\pi,m}(Y),p \rangle ,$ and by Lemma \ref{local structure}, $ \prod_{i=1}^{l}W_{r_{(i)}, a_{i}}^{g}\, \subseteq\, \langle W_{\pi,m}(Y), p \rangle.$
So, $ \prod_{i=1}^{l}W_{r_{(i)},a_{i}}^{g} $ is a union of irreducible componets of $ \langle W_{\pi,m}(Y),p \rangle.$ Since $ \pi^{\prime} \neq \pi,$ there exists $i_{0}$
 such that $ |r_{(i_{0})}|\,>\,1.$ By Lemma \ref{Singular}, $ W_{r_{(i_{0})},a_{i_{0}}}^{g}$ is singular, so $ \prod_{i=1}^{l}W_{r_{(i)},a_{i}}^{g}$ is singular. This forces $ \langle W_{\pi,m}(Y), p \rangle $ to be singular, that is, $\langle W_{\pi,m}(Y)$ is singular at $p$. 

 $(iii)$ We prove the first equality by induction on $t$. For $ t\,=\,0,$ it is clear. Suppose it is true for $t$. We have

\hspace{17pt}$ \textnormal{Sing}^{t+1}(S^{m}Y)\,=\, \textnormal{Sing}(\textnormal{Sing}^{t}S^{m}Y)$\\
 
\hspace{87pt} $=\, \textnormal{Sing}( \underset{|\pi|\,=\, m\,-\,t}{\bigcup{}}W_{\pi,m}(Y) )$ \\

\hspace{87pt} $=\, \big(\underset{|\pi|\,=\, m\,-\,t}{\bigcup}\textnormal{Sing\,}W_{\pi,m}(Y)\big )\, \cup\, \big(\underset{\underset{\pi_{1}\neq \pi_{2}}{|\pi_{1}|,|\pi_{2}|\,=\, m\,-\,t,}{}}\,\bigcup\, (W_{\pi_{1},m}(Y)\, \cap\, W_{\pi_{2},m}(Y)) \big )$\\

\hspace{89pt}$=\,\underset{\underset{\textnormal{for some } |\pi|\, =\, m \,-\, t}{\pi^{\prime}\medspace \textnormal{such that}\medspace \pi'\, >\, \pi}}{\bigcup} W_{\pi', m}(Y)$ \\

[the last equality holds by  $(ii)$ and noting that 
 $W_{\pi_{1},m}(Y)\, \cap\, W_{\pi_{2},m}(Y)\, =\, \underset{\underset{\pi^{\prime}\, >\, \pi_{2}}{\pi^{\prime}\,>\,\pi_{1}}}\bigcup W_{\pi', m}(Y)$]

\hspace{87pt} $=\,\underset{|\pi^{\prime}|\,\leq\, m\,-\,t\,-\,1}{\bigcup}W_{\pi^{\prime},m}(Y)$\\

\hspace{87pt}$=\, \underset{|\pi^{\prime}|\,=\,m\,-\,t\,-\,1}{\bigcup}W_{\pi^{\prime},m}(Y).$\\

Here the last equality holds as for any $|\pi^{\prime}|\,\leq\, m\,-\,t\,-\,1$, there  is $\pi^{\prime\prime}$ with $|\pi^{\prime \prime}|\,=\,m\,-\,t\,-\,1$ and $ \pi^{\prime} \,\geq\, \pi^{\prime\prime}$, so  $W_{\pi^{\prime},m}(Y) \,\subseteq\, W_{\pi^{\prime\prime},m}(Y).$

This proves the first equality in $(iii).$ The second equality follows immediately.

The part $(iv)$ follows immediately from $(iii),$ and the observation that for a partition $ \pi^{\prime}$ of $m$, we have $\pi^{\prime}\,>\,\pi$ for some $|\pi|\,=\,m\,-\,t$ $\iff  |\pi^{\prime}|\, \leq\, m\,-\,t\,-\,1.$ This proves the theorem.

\begin{remark}
In this remark $S^m \mathbb{C}^{n+1}$  denotes the $m$'th symmetric power of the vector space $\mathbb{C}^{n+1}$, which is a vector space of dimension ${m+n}\choose{n}$. Let $Y\,=\,(\mathbb{P}^n)^*$, the linear system of hyperplanes in $\mathbb{P}^n.$ Regard $\mathbb{P}(S^m\mathbb{C}^{n+1})$ as the linear system of degree $m$ hypersurfaces in $\mathbb{P}^n$. By {\cite[corollary 3.3.11, theorem 3.1.12]{Ry}}, $S^mY$ is embedded as a subvariety of $\mathbb{P}(S^m\mathbb{C}^n)$, whose points correspond to the degree $m$ effective divisors which are sums of hyperplanes. So, for $\pi\,=\,(a_1,...,a_k)$, points in $W_{\pi,m}^{\circ}(Y)$ correspond to effective divisors of the form $\sum_i a_i H_i$, where $H_i$'s are distinct hyperplanes. For positive integers $a_1, a_2, ..., a_k, d_1, d_2,...,d_k$ with $\sum_i a_id_i=m,$ if $W_{\underline{a}, \underline{d},m}^{\circ}$ is the locally closed subvariety of  $\mathbb{P}(S^m\mathbb{C}^{n+1})$ whose points correspond to effective divisors of the form $\sum_i a_i D_i$, where $D_i$'s are distinct prime divisors in $\mathbb{P}^n$ of degree $d_i$, then $W_{\underline{a}, \underline{d},m}^{\circ}$'s give a stratification of $\mathbb{P}(S^m\mathbb{C}^n)$ extending the stratification of $S^mY$ given in Theorem \ref{B}.
\end{remark}

The following corollary is the analogue of {\cite[Theorem 1.4]{Ha}} and {\cite[Theorem 1.4]{sosna2013fourier}} for symmetric powers.
 \begin{corollary}
     Let $m,n,p,q\,\geq\, 2$ be integers, $X,Y$ smooth varieties of dimension $m,n$ respectively. Suppose $S^pX\,\cong\, S^q Y$. Then $m\,=\,n, p\,=\,q, X\,\cong\, Y.$
 \end{corollary}
 \begin{proof}
     Let $\phi: S^pX\,\to\, S^q Y$ be an isomorphism. So, $\phi($Sing$^t (S^pX))\,=\,$ Sing$^t(S^qY)$ for all $t\,\geq\, 0.$ By Theorem \ref{B}, we have
     $$\textnormal{Sing}^{p\,-\,1} S^pX \,\cong\, X, \quad \textnormal{Sing}^p S^pX\,=\, \varnothing,$$ and $$\textnormal{Sing}^{q\,-\,1} S^qY \,\cong\, Y, \quad \textnormal{Sing}^q S^qY\,=\, \varnothing.$$ Hence, $p\,=\,q,\, X\,\cong\, Y.$ Comparing the dimensions of $S^pX$ and $S^pY$ one gets $pm\,=\,pn$, hence $m\,=\,n.$
 \end{proof}
 
\begin{remark}\label{tangent space}
    By Lemma \ref{local structure} and {\cite[Proof of Theorem 3.1.1]{Sh}}, it follows that for a partition $\pi\,=\,(a_1, a_2,...,a_k)$ of $m$ and $p\,\in\, W_{\pi, m}^0(Y),$ the dimension of the Zariski tangent space of $S^mY$ at $p$ is given by $\sum_{i=1}^k {{n+a_i}\choose {a_i}}\,-\,k.$
\end{remark}
\section{Discrepancy of symmetric power}

      In this section, we prove Theorem \ref{C}.
      
\textit{Proof of Theorem $\ref{C}$}:
 By, Lemma \ref{local structure}, we can assume $ Y \,=\,\mathbb{A}^n,$ and show the theorem for a neighborhood of $ \underline{0}\, \in\, S^{m} \mathbb{A}^n.$ By a similar argument as in {\cite[\S\, 3]{KX}}, for $ 1 \neq \sigma\, \in\, S_{m}$ we have $\textnormal{age}(\sigma)\, =\,\frac{m\,-\,\gamma}{2}\cdot n,$ where  $\gamma$ is the number of cycles in $ \sigma$. Minimum age  of $ \sigma \,\in\, S_{m}$ with  $ \sigma \neq 1$ is thus $ \frac{n}{2}.$ Now the theorem follows from Reid-Tai criterion as in \cite{IR}. In fact, the proof of the Reid-Tai criterion given there shows that in our situation,
$$ \textnormal{discrep}  \medspace S^m \mathbb{A}^n\,=\,\min\{\min_{1\, \neq\, \sigma\, \in\, S_m} \textnormal{age} (\sigma)-1, 1\} .$$
\section{Acknowledgements}\vskip 5mm We thank Prof. János Kollár and Prof. D. S. Nagaraj for their insightful discussions.
\printbibliography
 \end{document}